\documentclass[11pt, leqno]{article}
\usepackage[utf8]{inputenc}
\usepackage{lmodern}
\usepackage{subfiles}
\usepackage{enumitem}
	\setlist[enumerate]{label={\normalfont(\alph*)}} 
\usepackage{amsfonts}
\usepackage{amsthm}
\usepackage{amsmath}
\usepackage{amssymb}
\usepackage{amscd}
\usepackage{mathrsfs}
\usepackage{bbm}
\usepackage{multirow}
\usepackage{mathtools}
\usepackage{esint}

\usepackage{titlefoot}
\usepackage[margin=2.95cm]{geometry}
\usepackage{indentfirst}
\usepackage{graphicx}
\usepackage{graphics}
\usepackage{lscape}
\usepackage{tikz-cd}
\usepackage{color}
\usepackage{pict2e}
\usepackage{epic}
\usepackage{epstopdf}
\usepackage{titlesec}
	\titleformat{\section}[block]{\Large\bfseries\filcenter}{\thesection}{1em}{}
\usepackage{commath}
\usepackage{float}
\usepackage{caption}
\usepackage{subcaption}
\usepackage{etoolbox}
\usepackage[affil-it]{authblk}
\usepackage{combelow}

\let\oldbibliography\thebibliography
\renewcommand{\thebibliography}[1]{%
  \oldbibliography{#1}%
  \setlength{\itemsep}{-.5pt}%
}

\usepackage[hidelinks]{hyperref}
\hypersetup{bookmarksopen=true} 
\usepackage{hypcap}

\graphicspath{{./Pictures/}}
\allowdisplaybreaks

\usepackage{algorithm}
\usepackage[noend]{algpseudocode}

\usepackage{listings}
\usepackage{fancybox}

\renewcommand*\thesection{\arabic{section}}

\theoremstyle{plain}
\newtheorem{thm}{Theorem}
\newtheorem*{thm*}{Theorem}
\newtheorem{lemma}[thm]{Lemma}
\newtheorem*{lemma*}{Lemma}

\newtheorem*{cor*}{Corollary}

\theoremstyle{definition}

\newcommand{\thistheoremname}{}
\newtheorem{genericthm}[equation]{\thistheoremname}

\newcommand{\thistheoremnames}{}
\newtheorem*{genericthms}{\thistheoremnames}
\newenvironment{para*}[1]
  {\renewcommand{\thistheoremnames}{#1}%
   \begin{genericthms}}
  {\end{genericthms}}

\expandafter\let\expandafter\oldproof\csname\string\proof\endcsname
\let\oldendproof\endproof
\renewenvironment{proof}[1][\proofname]{%
  \oldproof[\upshape \bfseries #1:]%
}{\oldendproof}

\makeatletter
\def\@makechapterhead#1{%
  \vspace*{50\p@}%
  {\parindent \z@ \raggedright \normalfont
    \interlinepenalty\@M
    \Huge\bfseries  \thechapter.\quad #1\par\nobreak
    \vskip 40\p@
  }}
\makeatother

\def \a{\alpha}
\def \R {\mathbb{R}}

\def \C{\mathbb{C}}

\def \D{\textup{D}}
\def \e{\varepsilon}

\def \exc{\backslash}
\def \p{\partial}

\def \mc{\mathcal}

\def \tp{\textup}
\def \mb{\mathbb}

\newcommand{\imag}{\mathrm{i}}
\newcommand{\expo}{\mathrm{e}}

\makeatletter
\DeclareFontFamily{OMX}{MnSymbolE}{}
\DeclareSymbolFont{MnLargeSymbols}{OMX}{MnSymbolE}{m}{n}
\SetSymbolFont{MnLargeSymbols}{bold}{OMX}{MnSymbolE}{b}{n}
\DeclareFontShape{OMX}{MnSymbolE}{m}{n}{
    <-6>  MnSymbolE5
   <6-7>  MnSymbolE6
   <7-8>  MnSymbolE7
   <8-9>  MnSymbolE8
   <9-10> MnSymbolE9
  <10-12> MnSymbolE10
  <12->   MnSymbolE12
}{}
\DeclareFontShape{OMX}{MnSymbolE}{b}{n}{
    <-6>  MnSymbolE-Bold5
   <6-7>  MnSymbolE-Bold6
   <7-8>  MnSymbolE-Bold7
   <8-9>  MnSymbolE-Bold8
   <9-10> MnSymbolE-Bold9
  <10-12> MnSymbolE-Bold10
  <12->   MnSymbolE-Bold12
}{}

\let\llangle\@undefined
\let\rrangle\@undefined
\DeclareMathDelimiter{\llangle}{\mathopen}%
                     {MnLargeSymbols}{'164}{MnLargeSymbols}{'164}
\DeclareMathDelimiter{\rrangle}{\mathclose}%
                     {MnLargeSymbols}{'171}{MnLargeSymbols}{'171}
\makeatother

\begin{document}

 \title{\LARGE \textbf{On the necessity of the constant rank condition\\ for $L^p$ estimates}}

\author[1]{{\Large Andr\'e Guerra\vspace{0.4cm}}}
\author[2]{{\Large  Bogdan Rai\cb{t}\u{a}}}

\affil[1]{\small University of Oxford, Andrew Wiles Building, Woodstock Rd, Oxford OX2 6GG, United Kingdom \protect \\
{\tt{guerra@maths.ox.ac.uk}} \vspace{1em} \ }

\affil[2]{\small 
Max Planck Institute for Mathematics in the Sciences, Inselstraße 22, 04103 Leipzig, Germany\protect\\  {\tt{raita@mis.mpg.de}} \ }

\date{}

\maketitle

\begin{abstract}
We consider a generalization of the elliptic $L^p$-estimate suited for linear operators with non-trivial kernels. A classical result of Schulenberger and Wilcox (Ann. Mat. Pura Appl. (4) 88: 229–305, 1971) shows that if the operator has constant rank then the estimate holds. We prove necessity of the constant rank condition for such an estimate.
\end{abstract}

\vspace{0.2cm}

\unmarkedfntext{
\hspace{-0.85cm} 


\noindent  \emph{Acknowledgments:} The authors thank Jan Kristensen for encouragement and helpful comments. A.G. was supported by  [EP/L015811/1]. 
}

Consider a  linear constant-coefficient homogeneous differential operator $\mc A$, 
\begin{equation}
\label{eq:defA}
\mc A \varphi  = \sum_{|\alpha|=k} A_\a \p^\a \varphi, \qquad \varphi \colon \Omega\subseteq\R^n \to \mb V;
\end{equation}
here $\mb V, \mb W$ are finite-dimensional inner product spaces and $A_\a\in \tp{Lin}(\mb V,\mb W)$.
Given $1<p<\infty$, there is a constant $C_p$ such that
\begin{equation}
\Vert \D^k \varphi \Vert_{L^p(\R^n)} \leq C_p \Vert \mc A \varphi \Vert_{L^p (\R^n)} \qquad \tp{for all } \varphi \in C^\infty_c(\R^n, \mb V)
\label{eq:ellestimate}
\end{equation}
if and only if $\mc A$ is elliptic; this a classical result that goes back to the work of \textsc{Calderón}--\textsc{Zygmund} \cite{Calderon1952}. We recall that $\mc A$ is \emph{(overdetermined) elliptic} if the symbol
$\mb S^{n-1} \ni \xi \mapsto \mc A (\xi)\equiv \sum_{|\a|=k} (\imag\xi)^\a A_\a$
is injective. We also remark that the estimate \eqref{eq:ellestimate} only holds in trivial cases when $p=1$ \cite{Ornstein1962, Kirchheim2016} and $p=\infty$ \cite{Boman1972,Mityagin1958}.
We refer the reader to \cite{Faraco2020} for a short proof of the $p=1$ case in two dimensions.

Denote by $\mc F\equiv \widehat\cdot$ the Fourier transform and define for $\varphi \in C^\infty_c(\R^n,\mb V)$ the operator
$$\widehat{P_\mc A \varphi}(\xi) \equiv \tp{Proj}_{\ker \mc A(\xi)} \widehat\varphi(\xi).$$
All projections in this note are taken to be orthogonal. Note that $P_\mc A \varphi\in L^2(\R^n,\mb V)$ whenever $\varphi \in L^2(\R^n,\mb V)$.
The operator $\mc A$ is elliptic if and only if $P_\mc A=0$, a fact which explains the necessity of ellipticity for the estimate \eqref{eq:ellestimate}, c.f.\ \cite{VanSchaftingen2011}. 
Thus, it is natural to wonder whether \eqref{eq:ellestimate} holds if we test it only in the orthogonal complement of $\ker \mc A\subset L^2(\R^n,\mb V)$, and this is precisely what we investigate here.

The operator $\mc A$ has \emph{constant rank} if $\tp{rank}(\mc A(\xi))$ is constant for all $\xi \in \mb S^{n-1}$. In this note, we prove that this condition is equivalent to a more general version of \eqref{eq:ellestimate}:

\begin{thm*}
Given $1<p<\infty$, an operator $\mc A$ as in \eqref{eq:defA} has constant rank if and only if
\begin{equation}\Vert \D^k(\varphi - P_{\mc A} \varphi) \Vert_{L^p(\R^n)}
\leq C_p \Vert \mc A \varphi \Vert_{L^p(\R^n)} \qquad
\tp{ for all } \varphi \in C^\infty_c(\R^n, \mb V).
\label{eq:CRestimate}
\end{equation}
\end{thm*}

The sufficiency of the constant rank condition for the estimate \eqref{eq:CRestimate} is classical and seems to go back to the work of \textsc{Schulenberger}--\textsc{Wilcox} \cite{Schulenberger1971}, at least for the $p=2$ case, see also \cite{Kato1975,Murat1981}. It seems, however, that the necessity of this condition has remained unnoticed. 
 
The inequality \eqref{eq:CRestimate} is often used in the $L^p$-theory of Compensated Compactness \cite{Fonseca1999, Murat1981,Tartar1979} and, more recently, it has been used in \cite{Guerra2019} through constructions with potentials \cite{Raita2018}.
Moreover, when $p=1$ or $p= \infty$, \eqref{eq:CRestimate} never holds except in trivial cases: this is recovered from the classical results mentioned above by considering $\varphi = \mc A^*\psi$ for a test function $\psi$. On the other hand, strong type estimates on lower order derivatives in the spirit of \eqref{eq:CRestimate} can be proved, see \cite{Raita2018a}, building on \cite{Raita2018,VanSchaftingen2011}. Finally, we remark that the constant rank condition is not necessary for estimates on lower order derivatives, as can be seen from the simple example $\|u\|_{L^\infty}\leq \|\partial_1\partial_2 u\|_{L^1}$
for $u\in C_c^\infty(\R^2)$.

For $A\in \tp{Lin}(\mb V, \mb W)$, the \emph{Moore--Penrose generalized inverse} of $A$, sometimes called the \emph{pseudoinverse}, is the unique $A^\dagger\in\tp{Lin}(\mb W, \mb V)$ such that
$A A^\dagger = \tp{Proj}_{\tp{im}\, A}$ and 
$A^\dagger A =  \tp{Proj}_{\tp{im}\, A^*}.$ 
Equivalently, we may define $$A^\dagger \equiv \left( A|_{(\ker A)^\bot} \right)^{-1} \tp{Proj}_{\tp{im\,}A}.$$
We refer the reader to \cite{Campbell2009} for these and numerous other properties of generalized inverses.

The proof of the theorem is based on two observations, that we record as separate lemmas.

\begin{lemma}\label{lemma:smoothness}
Let $\Omega\subset \R^n$ be an open set. A smooth map $A\colon\Omega\to \tp{Lin}(\mb V, \mb W)$,  $A^\dagger\colon \Omega\to \tp{Lin}(\mb W, \mb V)$ is locally bounded if and only if $\tp{rank}\,A$ is constant in $\Omega$. In that case, $\mc A^\dagger$ is also smooth.
\end{lemma}


\begin{proof}
	Let $|\cdot|$ be the operator norm on $\tp{Lin}(\mb V, \mb W)$. We have that, for $\xi_1, \xi_2 \in \Omega,$
	\begin{equation}
	\tp{rank\,}(A(\xi_1))>\tp{rank\,}(A(\xi_2)) \quad \implies \quad |A^\dagger(\xi_1)| \geq \frac{1}{|A(\xi_1)-A(\xi_2)|}.\label{eq:daggerbound}
	\end{equation}
	Indeed, if the hypothesis holds then there exists $v\in \ker A(\xi_2)\cap \ker(A(\xi_1))^\bot$ with $|v|=1$. Thus $A^\dagger(\xi_1)(A(\xi_1)-A(\xi_2))v=A^\dagger(\xi_1) A(\xi_1)v=v$ and so
	$$1\leq |A^\dagger(\xi_1)(A(\xi_1)-A(\xi_2))|\leq |A^\dagger(\xi_1)||A(\xi_1)-A(\xi_2)|.$$
	Now suppose that $\tp{rank}\, A$ is not constant, so we can pick a point $\xi_0\in \Omega$ and a sequence $\xi_n\to \xi_0$ such that $\tp{rank\,}(A(\xi_n))\neq \tp{rank\,}(A(\xi))$. It follows from (\ref{eq:daggerbound}) that $A^\dagger$ is not bounded near $\xi_0$.

Conversely, assuming that $\tp{rank}\,A$ is constant, $A^\dagger$ is smooth. Indeed, and as in \cite{Raita2018}, this is easily deduced from \textsc{Decell}'s formula \cite{DecellJr.1965}
$$A^\dagger = -\frac{1}{a_r} A^* \left(\sum_{i=1}^{r} a_{i-1} (A A^*)^{r-i}\right),$$
where $r=\tp{rank}\, A$, $d=\dim\mb W$ and $p(\lambda)=(-1)^d \sum_{j=0}^d a_j \lambda^{d-j}$ is the characteristic polynomial of $A$; note that $a_j=0$ for $j>r$ and $a_r\neq 0$ away from zero. Since the coefficients $a_i$ depend polynomially on $A$, it follows that $A^\dagger$ is smooth.
\end{proof}

In order to deduce the theorem from Lemma \ref{lemma:smoothness}, we need the following auxiliary result:

\begin{lemma}\label{lemma:symbolbound}
If \eqref{eq:CRestimate} holds for some $1\leq p\leq \infty$, there is a constant $C$ such that
\begin{equation}
\label{eq:symbolbound}
|\xi|^k |\mc A^*(\xi) w|\leq C |\mc A(\xi)\mc A^*(\xi) w|
\qquad  \tp{ for all } w\in \mb W, \xi\in \R^{n}\exc\{0\}.
\end{equation}
\end{lemma}

An argument in a similar spirit, but concerning \eqref{eq:ellestimate}, is outlined in \cite{Boman1972}.

\begin{proof}
Fix $\xi\in \mb R^{n}\exc\{0\}$ and $w\in \mb W$ and let $g\in C^\infty_c(B_1(0))$ be such that $0\leq g\leq 1$ and $g=1$ in $B_{1/2}(0)$. Set $\varphi(x) \equiv \mc A^*(g(\e x) \expo^{\imag x\cdot \xi} w)$ for $\varepsilon\in(0,1)$, so that
\begin{align*}
\varphi_\e(x) & 
= g(\e x) \expo^{\imag x\cdot \xi} \mc A^*(\xi)w + \sum_{|\a|=k}\sum_{\beta<\alpha} \binom{\a}{\beta} \e^{k-|\beta|} (\imag\xi)^\beta \p^{\alpha-\beta} g(\e x)\expo^{\imag x\cdot\xi} A^*_\a w\\
& \equiv g(\e x) \expo^{\imag x\cdot\xi} \mc A^*(\xi)w+\e F_{\varepsilon}(x),
\end{align*}
where $F_\varepsilon\in C_c^\infty(\R^n,\mb V_\C)$ is supported inside $B_{1/\e}(0)$ and is bounded independently of $\e$ by $C_0(\mc A,g,\xi,w)$, say; here $\mb V_\C$ is the usual complexification of $\mb V$. On the other hand, 
$P_\mc A \varphi_\e=0$: indeed, $\ker \mc A(\xi)=\left(\tp{im}\,\mc A^*(\xi)\right)^\bot$ and so, writing $\eta(x) \equiv g(\e x)e^{i x\cdot \xi}w$,
$$\mc F(P_\mc A \varphi_\e) = \tp{Proj}_{\ker \mc A(\xi)} \mc A^*(\xi)\hat \eta(\xi)=0.$$
We can analogously obtain
\begin{align*}
\mc A \varphi_\e(x) & = g(\e x) \expo^{\imag x\cdot \xi} \mc A(\xi) \mc A^*(\xi) w
+ \varepsilon G_\varepsilon(x),
\end{align*}
where $G_\varepsilon\in C_c^\infty(\R^n, \mb W_\C)$ is supported inside $B_{1/\e}(0)$ and can be assumed to be bounded independently of $\e$ by $C_0$, so
\begin{equation}
\label{eq:boundA}
|\mc A\varphi_\e(x)|\leq |g(\e x)| |\mc A(\xi) \mc A^*(\xi)w| + \e\left|G_\varepsilon(x)\right|.
\end{equation}
A similar calculation yields
\begin{equation}
|\D^k \varphi(x)|\geq  |g(\e x)||\xi|^k|\mc A^*(\xi) w|-\e \left|H_\varepsilon(x)\right|
\label{eq:boundD}
\end{equation}
for another smooth function $H_\e$ having the same properties as $G_\e$. 
 Clearly we can assume that $\mc A^*(\xi) w\neq 0$ for otherwise there is nothing to prove. We take $\varepsilon$ small enough such that $|\xi|^k|\mc A^*(\xi)w|\geq C_0\e$, so  the right hand side of \eqref{eq:boundD} is non-negative inside $B_{1/(2\e)}(0)$. Thus, for $1\leq p<\infty$, combining \eqref{eq:boundA} and \eqref{eq:boundD} with \eqref{eq:CRestimate} we find
$$
\mathscr L^n(B_{1/(2\e)})\left(|\xi|^k |\mc A^*(\xi) w| 
- \e C_0\right)^p
\leq C\mathscr L^n(B_{1/\e})\left(
|\mc A(\xi) \mc A^*(\xi)w|+\e C_0 \right)^p.
$$
Dividing by $\mathscr L^n (B_{1/\e})$ and sending $\e\to 0$ we arrive at the conclusion. The case $p=\infty$ is similar, but easier.
\end{proof}

\begin{proof}[Proof of the theorem]
Note that, for any $\xi \in \R^n\exc \{0\}$, $\hat{\varphi}(\xi)-\tp{Proj}_{\ker \mc{A}(\xi)}\hat{\varphi}(\xi)=\mc A^\dagger(\xi)\widehat{\mc A \varphi}(\xi).$
Thus, by the definition of $P_\mc A$, we have that
\begin{equation*}\label{eq:fourierB}
\D^k(\varphi-P_{\mc A} \varphi)=\mc{F}^{-1}(\mc A^\dagger(\xi)\widehat{\mc A \varphi}(\xi)\otimes\xi^{\otimes k})
\end{equation*}
and the ``if" direction follows	from Lemma \ref{lemma:smoothness} and the H\"ormander--Mihlin multiplier theorem.
	
For the ``only if" direction, suppose that \eqref{eq:CRestimate} holds. Thus Lemma \ref{lemma:symbolbound} shows that \eqref{eq:symbolbound} must hold as well and this easily implies that $\mc A$ has constant rank. 
Indeed, \eqref{eq:symbolbound} shows that the spectrum of $\mc A(\xi)|_{\tp{im}\, \mc A^*(\xi)}$ is bounded away from zero uniformly in $\xi$; equivalently, 
$$\mb S^{n-1}\ni \xi \mapsto \left(\mc A(\xi)|_{\tp{im}\, \mc A^*(\xi)}\right)^{-1}
 \tp{ is  bounded.}
$$
The definition of $\mc A^\dagger$, together with Lemma \ref{lemma:smoothness}, show that $\mc A$ has constant rank.
\end{proof}
In fact, our observation can be improved when $p=2$:
\begin{cor*}
The operator $\mc A$ has constant rank if and only if there is a constant $C$ such that
\begin{equation}
\label{eq:closedrange}
\inf \left\{\Vert \D^k(\varphi -\psi)\Vert_{L^2(\R^n)}: \mc A \psi=0, \psi\in C^\infty_c(\R^n,\mb V)\right\} \leq C \Vert \mc A \varphi\Vert_{L^2(\R^n)}
\end{equation}
for all $\varphi \in C^\infty_c(\R^n, \mb V)$. In particular, $\mc A$ has constant rank if and only if the operator
$$\mc A\colon \mathscr{W}^{\mc A, 2}(\R^n)\equiv \tp{clos}_{\varphi\mapsto \Vert \mc A \varphi \Vert_{L^2}} C^\infty_c(\R^n,\mb V)\to L^2(\R^n,\mb W)$$
has closed range.
\end{cor*}

\begin{proof}
Note that the infimum in \eqref{eq:closedrange} is attained with $\psi=P_\mc A \varphi$, by Plancherel's theorem and the minimization properties of orthogonal projections. Hence the first part follows from the theorem, while the second statement is an immediate consequence of general results on unbounded linear operators, see for instance
\cite[§2.7, Remark 18]{Brezis2010}.
\end{proof}
Altogether, the observations made in the present note suggest that the general study of compensated compactness under linear partial differential constraints that are \emph{not of constant rank} requires substantially finer harmonic analysis tools, if any. Specifically, we refer to proving the results in \cite{Fonseca1999,Guerra2019,Murat1981} without any assumptions on the compensating differential operators.
In the particular case of quadratic forms \cite{Li1997, Tartar1979} or of simple operators \cite{Muller1999b} these assumptions can be bypassed but at present there is no general theory.

{\footnotesize
\bibliographystyle{acm}

\bibliography{/Users/antonialopes/Dropbox/Oxford/Bibtex/library.bib}
}

\end{document}